\documentclass[ journal,comsoc]{IEEEtran}
\usepackage[T1]{fontenc}

\ifCLASSINFOpdf
\else
\usepackage[dvips]{graphicx}
\fi
\usepackage{url}

\usepackage{amsmath}
\usepackage{amsthm}
\usepackage[cmintegrals]{newtxmath}

\usepackage{bm}
\usepackage{bbm}
\usepackage{subfigure}
\usepackage{graphicx}
\usepackage{calc}
\usepackage{epstopdf}
\newtheorem{theorem}{Theorem}
\newtheorem{lemma}{Lemma}
\newtheorem{remark}{Remark}

\newtheorem{corollary}{Corollary}
\newcommand{\norm}[1]{\left\lVert#1\right\rVert} 
\newtheorem{definition}{Definition}
\newtheorem{assumption}{Assumption}
\usepackage{verbatim}
\usepackage[vlined,linesnumbered,ruled,resetcount]{algorithm2e}
\usepackage{color}
\usepackage{array}
\usepackage{cite}
\usepackage{soul} 
\usepackage{xcolor}
\usepackage{floatrow}
\usepackage{lipsum}
\usepackage[bottom]{footmisc}
\usepackage{multicol}
\usepackage{multirow} 
\usepackage{algorithmic}

\begin{document}
\title{Privacy-preserving Incremental ADMM for Decentralized Consensus Optimization}

 \author{
	Yu~Ye, \IEEEmembership{Student~Member,~IEEE,}
	Hao~Chen, \IEEEmembership{Student~Member,~IEEE,}
    Ming~Xiao, \IEEEmembership{Senior~Member,~IEEE,}
	Mikael~Skoglund, \IEEEmembership{Fellow,~IEEE,}
	and~H.~Vincent~Poor \IEEEmembership{Fellow,~IEEE }

%
	
	\thanks{This work was supported in part by the U.S. National Science Foundation under Grant CCF-1908308. 
	
	Yu Ye, Hao Chen, Ming Xiao and Mikael Skoglund are with the School of Electrical Engineering and Computer Science, Royal Institute of Technology (KTH), Stockholm, Sweden (email: yu9@kth.se, haoch@kth.se, mingx@kth.se, skoglund@kth.se).
	
	H. Vincent Poor is with Department of Electrical Engineering, Princeton University, Princeton, USA (email: poor@princeton.edu).	
	}
		
}
  
\maketitle
 
\begin{abstract}
 
The alternating direction method of multipliers (ADMM) has been recently recognized as a promising optimizer for large-scale machine learning models. However, there are very few results studying ADMM from the aspect of communication costs, especially jointly with privacy preservation, which are critical for distributed learning. We investigate the communication efficiency and privacy-preservation of ADMM by solving the consensus optimization problem over decentralized networks. Since walk algorithms can reduce communication load, we first propose incremental ADMM (I-ADMM) based on the walk algorithm, the updating order of which follows a Hamiltonian cycle instead. However, I-ADMM cannot guarantee the privacy for agents against external eavesdroppers even if the randomized initialization is applied. To protect privacy for agents, we then propose two privacy-preserving incremental ADMM algorithms, i.e., PI-ADMM1 and PI-ADMM2, where perturbation over step sizes and primal variables is adopted, respectively. Through theoretical analyses, we prove the convergence and privacy preservation for PI-ADMM1, which are further supported by numerical experiments. Besides, simulations demonstrate that the proposed PI-ADMM1 and PI-ADMM2 algorithms are communication efficient compared with state-of-the-art methods.

\end{abstract}

\begin{IEEEkeywords}
  Decentralized optimization; alternating direction method of multipliers (ADMM); privacy preservation.  
\end{IEEEkeywords}
\IEEEpeerreviewmaketitle

\section{Introduction}

Consider a decentralized network $\mathcal{G}=(\mathcal{V},\mathcal{E})$, where $\mathcal{V}=\{1,...,N\}$ is the set of agents with computing capability, and $\mathcal{E}$ is the set of links. $\mathcal{G}$ is connected and there is no center agent. The agents seek to solve the consensus optimization problem (\ref{eq1}) collaboratively through sharing information among each other,
\begin{equation}\label{eq1}
\min_{x}~ \sum_{i=1}^{N}f_i(x;\mathcal{D}_i),
\end{equation}
where $f_i:\mathbbm{R}^p\to\mathbbm{R}$ is the local function at agent $i$, and $\mathcal{D}_i$ is the private dataset at agent $i$. All agents share a common optimization variable $x\in\mathbbm{R}^p$. The decentralized optimization problem (\ref{eq1}) is widely applied in many areas such as signal processing \cite{5444944,6945888,7041200}, machine learning \cite{predd2009,forero2010,5499155}, wireless sensor networks (WSNs) \cite{6197748,ID} and smart grids \cite{GB,HJ,WPG}, just to name a few.

In the existing literature, e.g., \cite{WPG,DGD,EXTRA,JFC,7423789,zhang2014,WADMM,PWADMM}, a large number of decentralized algorithms have been investigated to solve the consensus problem (\ref{eq1}). Typically, the algorithms can mainly be classified into primal and primal-dual types, namely, gradient descent based (GD) methods and the alternating direction method of multipliers (ADMM) based methods, respectively. In this work, we will use ADMM as the optimizer, which can usually achieve more accurate consensus performance than GD based methods with constant step size \cite{PWADMM}.

Though distributed gradient descent (DGD) \cite{DGD}, EXTRA \cite{EXTRA} and distributed ADMM (D-ADMM) \cite{JFC} have good convergence rates with respect to the number of iterations, these algorithms require each agent to collect information from all its neighbors. This makes the amount of communication high for each iteration. Hence, for the applications with unstable links among agents such as WSNs, the communication load becomes one of the main bottlenecks. Moreover the \textit{straggler problem} is also pronounced in synchronous ADMM based methods \cite{zhang2014}, where the total running time is significantly increased due to slowly computing nodes. Besides, for the ADMM based decentralized approaches provided in \cite{zhang2014,JFC,7423789,WADMM,PWADMM}, agents are required to exchange primal and dual variables with neighbors in each iteration. This inevitably leads to the privacy problem if the local function $f_i$ and data set $\mathcal{D}_i$ are private to agent $i(\in\mathcal{N})$, for instance, in the practical applications with sensitive information such as personal medical or financial records, confidential business notes and so on. Thus, guaranteeing privacy is critical for decentralized consensus.
 

\subsection{Related Works}


To reduce the communication load for solving consensus problem (\ref{eq1}), various decentralized approaches have been proposed recently. Among these efforts, one important direction is to limit information sharing for each iteration. In \cite{7552562}, given an underlying graph, the weighted ADMM is developed by deleting some links prior to the optimization process. Communication-censored ADMM (COCA) in \cite{COCA} can adaptively determine whether a message is informative during the optimization process. Following COCA, an extreme scenario is to only activate one link in the graph per iteration such as the random-walk ADMM (W-ADMM) \cite{WADMM}, which incrementally updates the optimization variables. To achieve optimized trade-off between communication cost and running time, parallel random walk ADMM method (PW-ADMM) and intelligent PW-ADMM (IPW-ADMM) are proposed in \cite{PWADMM}. References \cite{4276978,5495951} also take advantages of random walks but deal with stochastic gradients. Another line of works is to exchange sparse or quantized messages to transmit fewer bits. Following this direction, the quantized ADMM is provided in \cite{7472455}. Two quantized stochastic gradient descent (SGD) methods, sign SGD with error-feedback (EF-SignSGD) and periodic quantized averaging SGD (PQASGD), are proposed in \cite{pmlrv97} and \cite{NIPS20181}, respectively. Aji et al. \cite{fikri2017sparse} presented a heuristic approach to truncate the smallest gradient components and only communicate the remaining large ones. However, these algorithms cannot guarantee exact convergence to the optimal solution \cite{8558107}. 


Meanwhile, with increasing concerns on data privacy, e.g., GDPR in Europe \cite{EUdataregulations2018}, more and more efforts are allocated to developing privacy preserving algorithms for solving the decentralized consensus problem (\ref{eq1}). To measure the privacy preserving performance, differential privacy has been used \cite{xiao2019local}. In general, artificial uncertainty is introduced over shared messages or local functions to maintain sensitive information secure against eavesdroppers. Differential private distributed algorithms based on GD have been presented in \cite{huang2015differentially,han2016differentially,nozari2016differentially}, where noise following different distributions is added to variable or gradients. Moreover, in primal-dual methods \cite{zhang,zhang18f,7563366,8772211}, perturbation is conducted either on penalty term, or on the primal and dual variables before sharing to neighboring agents. 
On the other hand, cryptographic based methods, such as (partial) homomorphic Encryption \cite{alexandru2018cloud,5156499}, are also adopted to protect the privacy of agents through applying the encryption mechanism to exchanged information. Other privacy-preserving approaches, including perturbing problems or states, are provided in \cite{mangasarian2012privacy,gade2017private}. However, these methods lead to high computational overheads and communication load especially in large-scale decentralized optimization.  



\subsection{Motivation and Contributions}

W-ADMM can significantly reduce communication load by activating only one node and link per iteration in a successive manner. Since the sequence of the updating order for agents is randomized by following a Markov chain, it is possible that the primal and dual variables at some agents are updated for much fewer rounds than others. Thus, the convergence speed may degrade. To guarantee agents not to be inactive for a long time, \textit{Random Walk with Choice} (RWC) is introduced in IPW-ADMM \cite{PWADMM} to reduce running time. The main principle of RWC is that through restricting the updating order of agents, the convergence speed can be improved for incremental approaches, which shall save communication load. This phenomenon is also supported by the walk proximal gradient (WPG) presented in \cite{WPG}, where the updating order of agents follows a \textit{Hamiltonian Cycle}. Therefore, in what follows, we will fix the updating order of agents as WPG and present the incremental ADMM (I-ADMM). Different from WPG, the I-ADMM is a primal-dual based approach. Moreover, comparing to PW-ADMM and IPW-ADMM, there is only one walk with a fixed order in I-ADMM.

To enable the privacy preservation for I-ADMM against the eavesdropper, which overhears all links in $\mathcal{E}$, we introduce uncertainty among local primal-dual variables and the transmitted tokens, e.g., through adding artificial noise over primal or dual variables \cite{zhang,7563366}. However, without scaling noise with iterations, there exists an error bound at convergence \cite{7563366}, which leads to a trade-off between privacy and accuracy. To avoid compromising convergence performance, we propose to perform perturbation over the step size for both updates of primal and dual variables, i.e. $\{x_i,y_i  \}_{i\in\mathcal{V}}$. Equivalently, this procedure can be seen as adding noise, which is scaled by the primal and dual residues, over $x_i^k$ and $y_i^k$, respectively for iteration $k$.

Different from the existed results, in what follows, we will study the ADMM based optimizer for solving (\ref{eq1}) from the aspect of both communication efficiency and privacy preservation. The main contributions of this paper can be summarized as follows.  
 \begin{itemize}

	\item We propose the I-ADMM method for solving decentralized consensus optimization (\ref{eq1}). Different from other incremental ADMM methods such as W-ADMM \cite{WADMM}, the update order of I-ADMM follows a Hamiltonian cycle, which can further reduce communication load and improve the convergence speed. 

    \item Since I-ADMM is not privacy-preserving, we first introduce the randomized initialization. To guarantee privacy preservation, we provide two privacy-preserving incremental ADMM (PI-ADMM) algorithms, i.e., PI-ADMM1 and PI-ADMM2, which apply perturbation over step size and primal variable, respectively. To the best of our knowledge, this is the first paper to consider privacy preservation for incremental ADMM.
    
    \item We prove the convergence of PI-ADMM1, and show that at convergence, the primal and dual variables generated by PI-ADMM1 satisfy the Karush-Kuhn-Tucker (KKT) conditions with mild assumptions over local functions. Besides, we prove that privacy preservation is guaranteed against eavesdropper or colluding agents.

 	\item Numerical results show that the proposed I-ADMM based algorithms are communication efficient compared with state-of-the-art methods. Moreover, we show that PI-ADMM1 can guarantee both privacy-preservation and accurate convergence.     
 
 \end{itemize}

 The remainder of this paper is organized as follows. We first introduce the incremental ADMM in Section II. To guarantee the privacy of local agents against external eavesdroppers, we present two PI-ADMM algorithms in Section III. The convergence and privacy analyses regarding proposed approaches are presented in Section IV. To validate the efficiency of proposed methods, we provide numerical experiments in Section V. Finally, Section VI concludes the paper.
 
 \section{Incremental ADMM}

 
By defining $ \bm{x}=[x_1,...,x_N]\in\mathbbm{R}^{pN} $, problem (\ref{eq1}) can be rewritten as
\begin{equation}\label{eq2}
\begin{aligned}
\min_{\bm{x}, z}~\sum_{i=1}^{N}f_i(x_i;\mathcal{D}_i),  ~~~
s.t.~  \mathbbm{1}\otimes z-\bm{x}=\bm{0},
\end{aligned}
\end{equation}
where $z\in\mathbbm{R}^p$, $\mathbbm{1}=[1,...,1]\in\mathbbm{R}^{N}$, and $\otimes$ is Kronecker product. For simplicity, we denote $f_i(x_i;\mathcal{D}_i)$ as $f_i(x_i)$. The augmented Lagrangian for problem (\ref{eq2}) is 
 \begin{figure}[t] 
	\vskip 0.2in
	\begin{center}
		\centerline{\includegraphics[width=76mm]{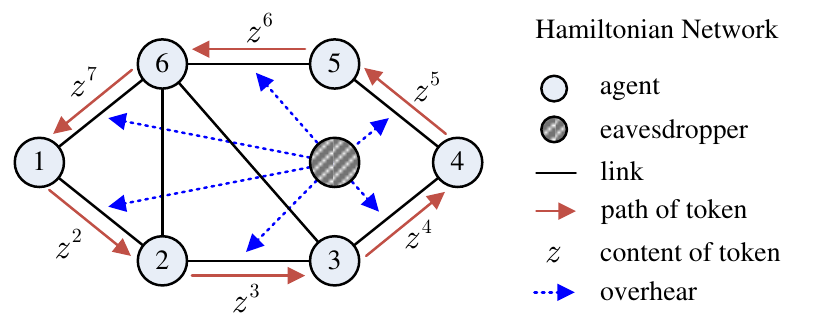}}
		\caption{An example of I-ADMM with an external eavesdropper.}
		\label{1}
	\end{center}
	\vskip -0.2in
\end{figure}

\begin{figure}[t] 
	\vskip 0.2in
	\begin{center}
		\centerline{\includegraphics[width=88mm]{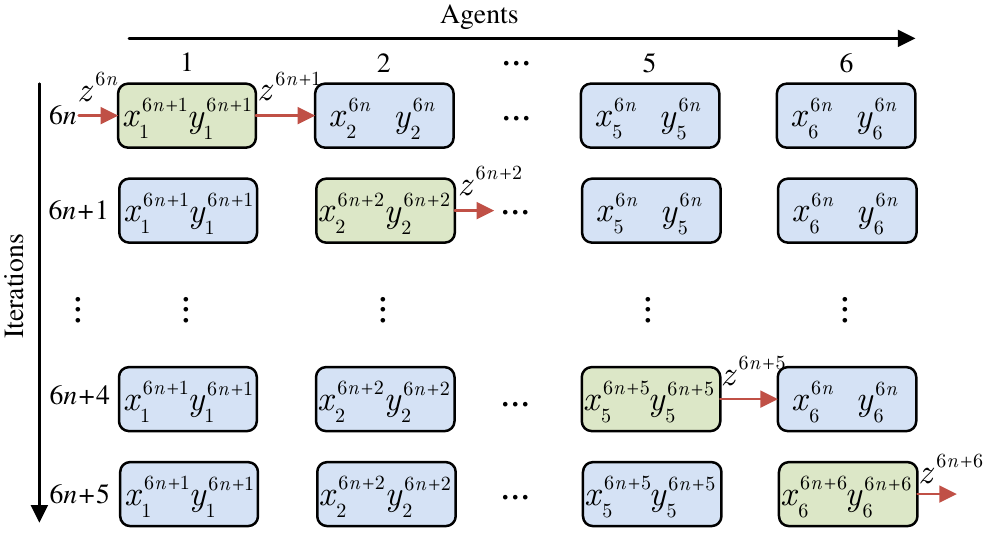}}
		\caption{Updating process of I-ADMM for example in Fig. 1 with iterations $[6n,6n+5]$ where $n\in\mathbbm{N}$.}
		\label{1}
	\end{center}
	\vskip -0.2in
\end{figure}
  \begin{algorithm}[t]
	\caption{I-ADMM } 
	\begin{algorithmic}[1]
		\STATE \textbf{initialize}: $\{z^0 = \bm{0}, x_i^0 =\bm{0}, y_i^0=\bm{0}, k=0 |i\in\mathcal{V}\}$; 
		\FOR{$k=0,1,...$}
		\STATE agent $i_k = k \mod N + 1$ do: 		 
		\STATE \textbf{receive} token $z^{k }$;   
		\STATE \textbf{update} $\bm{x}^{k+1}$ according to (\ref{eq4b});
		\STATE \textbf{update} $\bm{y}^{k+1}$ according to (\ref{eq4c});
		\STATE \textbf{update} $z^{k+1}$ according to (\ref{eq6});
		\STATE \textbf{send} token $z^{k+1} $ to agent $i_{k+1}=k \mod N + 2$;  
		\ENDFOR 		
	\end{algorithmic} 
\end{algorithm} 
\begin{equation}
\mathcal{L}_{\rho}(\bm{x},\bm{y},z)=\sum _{i=1}^{N}f_i(x_i) + \left\langle \bm{y} ,\mathbbm{1}\otimes z-\bm{x}   \right\rangle + \frac{\rho}{2}\norm{\mathbbm{1}\otimes z-\bm{x}}^2,
\end{equation}
where $ \bm{y}=[y_1,...,y_N]\in\mathbbm{R}^{pN}$ is the dual variable, while $\rho>0$ is a constant parameter. Following W-ADMM \cite{WADMM}, with guaranteeing $\sum_{i=1}^{N}(x_i^0-\frac{y_i^0}{\rho})=\bm{0}$, the updates of $\bm{x}$, $\bm{y}$ and $z$ at the $(k+1)$-th iteration are given by
  \begin{subequations}
	\begin{align}		
	&x_i^{k+1}:=\left\{\begin{aligned}
	&\arg \min_{x_i} f_i(x_i) + \frac{\rho}{2} \left\|z^{k }-x_i+\frac{y_i^k}{\rho} \right\|^2 ,~i=i_k ;\\
	&x_i^{k },~i\in\mathcal{V},i\neq i_k;
	\end{aligned}  \right. \label{eq4b}\\
	&y_i^{k+1}:=\left\{\begin{aligned}
	& y_i^{k} + \rho \left( z ^{k}-x_{i}^{k +1} \right),~i=i_k ;\\
	& y_i^{k },~i\in\mathcal{V},i\neq i_k ;\label{eq4c}
	\end{aligned}  \right. \\
	& z ^{k+1}:=\frac{1}{N}\sum_{i=1}^{N} \left(x_i^{k+1}-\frac{y^{k+1}_i}{\rho} \right).\label{eq4a}  
	\end{align}	
\end{subequations}
Note that the updates (\ref{eq4b})-(\ref{eq4a}) are not decentralized since (\ref{eq4a}) collects information from all agents. With initializing $\bm{x}^0=\bm{0}$ and $\bm{y}^0=\bm{0}$, it is easy to verify that the update of $z$ can be \textit{incrementally} implemented as
\begin{equation}
 z^{k+1} = z^{k } + \frac{1}{N} \left[ \left(x_{i_k}^{k+1}-\frac{y_{i_k}^{k+1}}{\rho}  \right) -  \left(x_{i_k}^{k }-\frac{y_{i_k}^{k }}{\rho}  \right)  \right].\label{eq6}
\end{equation}
Then the decentralized implementation of I-WADMM is presented in Algorithm 1. Different from W-ADMM, the agents are activated in a predetermined circulant pattern for I-ADMM, where the activated agent in the $k$-th iteration is $i=k\mod N + 1$. Furthermore, we assume walk $\langle i_k \rangle_{k\geq 0}$ repeats a Hamiltonian cycle order: $1\to 2\to \cdots \to N \to 1\to 2\to \cdots$ as WPG \cite{WPG}. In I-ADMM, the variable $z^{k+1}$ gets updated at agent $i_k$ and passed as a token to its neighbor $i_{k+1}$ through  Hamiltonian cycle. In Fig. 2, we present the evolution of $\{x_i,y_i\}$ at agents and the transition of token $z$ when I-ADMM is applied to the example shown in Fig. 1.


\subsection{Convergence Analysis}

 To start with the convergence analysis for I-ADMM, we first make the following assumptions over network $\mathcal{G}$ and local functions $\{f_i\}$. For simplicity, we denote $F(\bm{x})=\sum_{i=1}^{N}f_i(x_i)$.

\begin{assumption} 
Graph $\mathcal{G}$ is well connected and there exists a Hamiltonian cycle.
\end{assumption}

\begin{assumption}
	The objective function $F(\bm{x})$ is bounded from below over $\bm{x}$, and $F(\bm{x})$ is coercive over $\bm{x}$. Each $f_i(x)$ is L-Lipschitz differentiable, i.e., for any $u,v\in\mathbbm{R}^p$, 
	\begin{align}
		\left\| \nabla f_i(u)-\nabla f_i(v)  \right\| \leq L \left\|u-v  \right\|, ~i=1,...,N. 
	\end{align}
\end{assumption}

With Assumption 2, we obtain the following useful inequality for loss function $ f_i$.
\begin{lemma}
	For any $u,v\in \mathbbm{R}^p$, we have
	\begin{equation}\label{eq7}
		f_i(u)\leq f_i(v) + \left\langle \nabla f_i(v), u-v \right\rangle +\frac{L}{2}\|u-v \|^2.
	\end{equation}
\end{lemma}
\begin{proof}
 Inequality (\ref{eq7}) is well known, see \cite{Ortega}.
\end{proof}
Then according to Lemma 1 and Theorem 1 in \cite{WADMM}, we have the following convergence property for I-ADMM.
\begin{lemma}\label{theorem1}
	Under Assumptions 1 and 2, for $\rho \geq 2L+2$, the iterations $(\bm{x}^k,\bm{y}^k,z^k)$ generated by I-ADMM satisfy the following properties:
	\begin{enumerate}
		\item $\mathcal{L}_{\rho}(\bm{x}^k,\bm{y}^k,z^k)-\mathcal{L}_{\rho}(\bm{x}^{k+1},\bm{y}^{k+1},z^{k+1})\geq 0$;
		\item $\langle \mathcal{L}_{\rho}(\bm{x}^k,\bm{y}^k,z^k) \rangle_{k\geq0} $ is lower bounded and convergent;
		\item $\lim\limits_{k\to \infty} \left\| \nabla \mathcal{L}_{\rho}(\bm{x}^k,\bm{y}^k,z^k)\right\|=0 .$
	\end{enumerate}
\end{lemma}
 \begin{proof}
 	The convergence proof of I-ADMM is similar to that of W-ADMM in \cite{WADMM}. The difference is that as Assumption 1 in \cite{WADMM}, each agent is visited with a fixed period, i.e. $\mathcal{J}(\delta)=N$.
 \end{proof}
 
\subsection{Computation Efficient I-ADMM}
 With complex local function $f_i$, the computation load for solving minimization problem (\ref{eq4b}) is high. To alleviate computation burden for the update of $\bm{x}^{k+1}$ in I-ADMM, we can apply the first-order approximation by replacing (\ref{eq4b}) with update:
 
	\begin{align}	 
	&x_i^{k+1}:=\left\{\begin{aligned}
	&z^{k } + \frac{y_i^k}{\rho}-\frac{1}{\rho}\nabla f_i\left(x_i^k  \right),~i=i_k ;\\
	&x_i^{k },~i\in\mathcal{V},i\neq i_k.
	\end{aligned}  \right. \label{eq4b11} 
	\end{align}	
 Comparing to (\ref{eq4b}), update process (\ref{eq4b11}) can save computation. However, it may cause higher communication loads to achieve the same accuracy, as shown in \cite{ye2019decentralized}. In what follows, we will focus on investigating the privacy preserving I-ADMM based on Algorithm 1. Besides, the obtained results can also be extended to the computation efficient I-ADMM methods with privacy preservation. 
 
\section{Incremental ADMM with Privacy Preservation}
 
In this section, we will first analyze the privacy of Algorithm 1. Then to guarantee the security for agents in solving (\ref{eq2}), we propose the PI-ADMM algorithms.

\subsection{Review of I-ADMM} 
  Assume that there exists an external eavesdropper overhearing all the links, as shown in Fig. 1. While the agents want to protect their private functions and datasets from the eavesdropper. The privacy property of I-ADMM is analyzed as follows.
 \begin{lemma}\label{lemma1}
 	In Algorithm 1, the intermediate states and gradients $\{x_{i }^{k+1},y_{i }^{k+1},\nabla f_{i }(x_{i }^{k +1})|i\in\mathcal{V}, k =0,1,... \}$ can be inferred by the eavesdropper. 
 \end{lemma}
 \begin{proof}
 	 Assume that the eavesdropper has the knowledge of $x_{i_k}^{k }$ and $y_{i_k}^{k }$ where $k \geq 0$. Through the update process (\ref{eq4b}), the eavesdropper can establish the following equations
 	\begin{align} 
 		\nabla f_{i_k} \left(x_{i_k}^{k +1} \right)&= y_{i_k}^{k +1} .\label{eq11}
 	\end{align}
 	Hence once $y_{i_k}^{k +1}$ is inferred by the master, $\nabla f_{i_k} (x_{i_k}^{k +1} )$ will be revealed as well. Combining the update process (\ref{eq4c}) and (\ref{eq6}), we have
 	\begin{equation}\label{eq13} 	
 	\left\{ 
 	\begin{aligned} 	
 		&y_{i_k}^{k +1} = y_{i_k}^{k } + \rho \left(z^{ k   } - x_{i_k}^{k +1} \right); \\
 		&\Delta^{k+1} = \frac{1}{N}\left[ \left(x_{i_k}^{k +1}-\frac{y_{i_k}^{k +1}}{\rho }  \right)- \left(x_{i_k}^{k}-\frac{y_{i_k}^{k}}{\rho } \right) \right],
 		\end{aligned}	
 		\right.		
 	\end{equation} 
  where $\Delta^{k+1}=z^{k+1}-z^{k }$. Through rearranging (\ref{eq13}), the recursive equations are given by
 	\begin{equation}\label{eq12}
 		\left\{ 
 		\begin{aligned}
 		x_{i_k}^{k +1} =& \frac{1}{2} \left(  N \Delta^{k+1}+ z^{k } + x_{i_k}^{k} \right);\\
 		y_{i_k}^{k+1} =& y_{i_k}^{k}+\frac{\rho}{2} \left(  z^{k } - N \Delta^{k+1} -x_{i_k}^{k } \right) ,
 		\end{aligned}
 		 \right.
 	\end{equation}  
 	 Since the eavesdropper also overhears $z^{k +1}$ and $z^{k }$, (\ref{eq12}) consists of 2 equations with 2 variables. Hence $x_{i_k}^{k +1}$ and $y_{i_k}^{k+1}$ can be obtained by solving (\ref{eq12}). In I-ADMM, since $\{x_i^0,y_i^0|i\in\mathcal{V}\}$ are known to the eavesdropper, with (\ref{eq11}) and the update (\ref{eq4b}), (\ref{eq4c}), we can conclude that $\{x_i^{k },y_i^{k },\nabla f_i(x_i^{k })|i\in\mathcal{V},k =0,1,...  \}$ can also be inferred. 
 \end{proof}
In addition, with $\{\nabla f_i(x_i^{k})|i\in\mathcal{V},k =0,1,... \}$, the private functions $\{f_i \}$ can be inferred by the eavesdropper \cite{zhang}.
 

 \begin{algorithm}[t]
 	\caption{I-ADMM with Random Initialization} 
 	\begin{algorithmic}[1]
 		\STATE \textbf{initialize}: $z^1=\bm{0} $;
 		\FOR{$k=0,...,N-1$}	 		 
 		\STATE agent $i  = k \mod N + 1$ do:	 
 		\STATE \textbf{generate} $ \upsilon\in\mathbbm{R}^p$ randomly;
 		\STATE \textbf{set} $x_i^0\gets \upsilon$ and $y_i^0\gets -\rho \upsilon $;
 		\ENDFOR				
 		\STATE \textbf{run} steps 2-9 of Algorithm 1.
 	\end{algorithmic} 
 \end{algorithm}

\subsection{Privacy-preserving I-ADMM}

  Inspired by the privacy-preservation definitions in \cite{zhang,yan2013,lou2018}, we define the privacy as follows.     
\begin{definition}
	A mechanism $\mathcal{M}: \mathcal{M}(\mathcal{X})\to   \mathcal{Y}$ is defined to be privacy preserving if the input $\mathcal{X}$ cannot be uniquely derived from the output $\mathcal{Y}$.
\end{definition}

\begin{algorithm}[t]
	\caption{PI-ADMM1 with Step Size Perturbation} 
	\begin{algorithmic}[1]
		\STATE Follow steps of Algorithm 2, but substitute steps 6 and 7 of Algorithm 1 with the following 4 steps; 
		\STATE \textbf{generate} $\gamma_{i_k}^{k}\sim \mathbbm{P}_{i_k}$;
		\STATE \textbf{set} $\tilde{\rho}_{i_k}^{k } \gets \rho  \cdot \gamma_{i_k}^{k} $;
		\STATE \textbf{update} $\bm{x}^{k+1}$ by 
		\begin{align}		
		&x_i^{k+1}:=\left\{\begin{aligned}
		&\arg \min_{x_i} f_i(x_i) + \frac{\tilde{\rho}^k_i }{2} \left\|z^{k }-x_i+\frac{y_i^k}{\tilde{\rho}^k_i } \right\|^2 ,~i=i_k ;\\
		&x_i^{k },~\text{o.w.};
		\end{aligned}  \right. \label{eq4a11}  
		\end{align}	
		\STATE \textbf{update} $\bm{y}^{k+1}$ by 
		\begin{align}			    
		&y_i^{k+1}:=\left\{\begin{aligned}
		& y_i^{k} + \tilde{\rho}^k_i \left( z ^{k}-x_{i}^{k +1} \right),~i=i_k ;\\
		& y_i^{k },~\text{o.w.} .\label{eq4c1}
		\end{aligned}  \right. 
		\end{align}		
	\end{algorithmic} 
\end{algorithm}

 \begin{algorithm}[t]
 	\caption{PI-ADMM2 with Primal Perturbation } 
 	\begin{algorithmic}[1]
 		\STATE Follow steps of Algorithm 2, but substitute step 5 of Algorithm 1 with the following two steps; 
 		\STATE \textbf{generate} $\omega_{i_k}^k\sim \mathcal{N}(0,\sigma )$;
 		\STATE \textbf{update} $\bm{x}^{k+1}$ by 
 		\begin{align} 
 		&x_i^{k+1}:=\left\{\begin{aligned}
 		&\arg \min_{x_i} f_i(x_i) + \frac{\rho}{2}\left\|z^{k}-x_i+\frac{y_i^k}{\rho} \right\|^2 +  \omega_i^k, ~i=i_k ;\\
 		&x_i^{k },~\text{o.w.}.
 		\end{aligned}  \right. \label{eq4a1} 
 		\end{align}	
 		
 	\end{algorithmic} 
 \end{algorithm}

Hence from Lemma 1, I-ADMM is not privacy preserving. Recalling the proof of Lemma \ref{lemma1}, by making $x_i^0$ and $y_i^0$ private, information $\{x_i^{k},y_i^{k},\nabla f_i(x_i^{k})|k=0,1,...\}$ cannot be inferred by the eavesdropper exactly from the recursive equation (\ref{eq12}) if the states and multipliers are unknown at the end of iteration. In fact, to implement the incremental update (\ref{eq6}), we only need to ensure $\sum_{i=1}^{N}(x_i^0-\frac{y_i^0}{\rho})=\bm{0}$. To introduce uncertainty, we randomize the initialization and present Algorithm 2, where variables $x_i^0$ and $y_i^0$ are randomly generated at agent $i$ such that (\ref{eq4a}) is satisfied. Then Algorithm 1 is carried out with initialization $\{ x_i^0=\upsilon ,y_i^0=-\rho \upsilon |i\in\mathcal{V}\}$.

To enhance the privacy of Algorithm 2, we can introduce more uncertainty by applying perturbation on primal and dual updates. Through substituting the update of primal and dual variables with (\ref{eq4a11}) and (\ref{eq4c1}), respectively, where the step size $\rho$ is multiplied by a perturbation $\gamma_{i_k}^k$, we can obtain the PI-ADMM1 in Algorithm 3. It is worth noting that (\ref{eq4c1}) can be regarded as introducing additive noise, i.e., $\rho(\gamma_{i_k}^k-1)(z^{k+1}-x_{i_k}^{k+1})$, over the dual variable updating, which is scaled by the primal residue. Regarding to the primal update with the first order approximation, it can also be expressed as appending the scaled noise $\frac{1}{\rho}(\frac{1}{\gamma_{i_k}^k}-1)(y_{i_k}^k-\nabla f_{i_k}(x_{i_k}^k))$ over $x_{i_k}^{k+1}$. Different from existing ADMM based privacy-preserving methods \cite{zhang,xiao2019local}, we randomly generate initialization in PI-ADMM1 instead of implementing encryption for each iteration. Thus, the computation complexity and communication load of PI-ADMM1 are significantly reduced. In addition, the applied perturbation mechanism avoids computation over step size $\rho$ as Algorithm 1 in \cite{zhang}.

From the results of \cite{xiao2019local}, another approach to enhance the privacy of PI-ADMM1 is perturbing the primal updates by adding noise according to fixed stochastic, which is illustrated in Algorithm 4. For the $k$-th iteration, an artificial noise $\omega_{i_k}^k$ generated at agent $i_k$ locally is affiliated over $x_{i_k}^{k+1}$. From the analyses in \cite{xiao2019local}, this perturbation mechanism can also guarantee the privacy for agents. However, due to page limitation, we only evaluate the convergence property for PI-ADMM2 in Section V. 

\section{Convergence and Privacy Analysis}
In this section, we will first prove the convergence for PI-ADMM1 and PI-ADMM2. Then privacy analysis is provided.
\subsection{Convergence Analysis}
We first show the convergence of Algorithm 2. 
\begin{remark}
	Algorithm 2 has the same convergence properties as those of I-ADMM.
\end{remark}
\begin{proof}
	Since $\sum_{i=1}^{N}(x_i-\frac{y_i}{\rho})=\bm{0}$ is guaranteed from steps 2-6 in Algorithm 2, the convergence can be proved as Lemma 2.
\end{proof}

Denoting $x^*=\arg \min_x F(x)$, then with Assumption 2, The convergence properties of PI-ADMM1 is provided as follows.
\begin{lemma}\label{lemma2}
	With the following conditions
	\begin{subequations}
		\begin{align}
		\rho &>L; \\\gamma_{i_k}^k&> \max\left\{\frac{2\rho^2 + 4\rho +1}{\rho-L} ,2(\rho+2)N  \right\}, k=0,1,...,
		\end{align}
	\end{subequations}	
the iterations $(\bm{x}^k,\bm{y}^k,z^k)$ generated by PI-ADMM1 satisfy the following properties:
	\begin{enumerate}
		\item $\mathcal{L}_{\rho}(\bm{x}^k,\bm{y}^k,z^{k })-\mathcal{L}_{\rho}(\bm{x}^{k+1},\bm{y}^{k+1},z^{k+1})\geq 0$;
		\item $\langle \mathcal{L}_{\rho}(\bm{x}^k,\bm{y}^k,z^{k }) \rangle_{k\geq0} $ is lower bounded by optimum $F(x^*)$ and convergent;
	\end{enumerate}
\end{lemma}
\begin{proof}
	See Appendix A.
\end{proof}

Based on the analysis of Lemma \ref{lemma2} for PI-ADMM1, the convergence property is summarized in the following theorem.

\begin{theorem}\label{theorem111}
	Following Assumptions 1 and 2, and the conditions given in Lemma \ref{lemma2}, the iteration $(\bm{x}^k,\bm{y}^k,z^k)$ generated by PI-ADMM1 have limit points, which satisfy the KKT conditions of problem (\ref{eq2}).
\end{theorem}

\begin{proof}
	See Appendix B.
\end{proof}

It should be noted that the conditions given in Lemma \ref{lemma2} and Theorem \ref{theorem111} is sufficient for the convergence of PI-ADMM1.  

\subsection{Privacy Analysis}

In what follows, we will analyze the privacy property for Algorithm 2 and PI-ADMM1.

\begin{remark}\label{theorem21}
	In Algorithm 2, the states, multipliers and gradients, i.e. $\{x_{i_k}^{k+1}, y_{i_k}^{k+1}, \nabla f_{i_k}(x_{i_k}^{k+1})| i\in\mathcal{N},k=0,..., K  \}$, cannot be inferred by the eavesdropper exactly if $z^{k+1}\neq x_{i_k}^{k+1}$ for all $k=0,...K$.
\end{remark}
\begin{proof}
	Assume that the eavesdropper collects information from $K$ iterations to infer the information of agents. According to (\ref{eq12}), the measurements considering all the agents can be formulated as  
	\begin{equation}\label{eq1712}
	\left\{  
	\begin{aligned}
	& x_{i }^0 - \frac{y_{i }^0}{\rho} =\bm{0},~i\in\mathcal{V};\\
	&x_{i_0}^{ 1} =  \frac{1}{2}\left(  N\Delta^1+ z^{0} + x_{i_0}^{0}\right);\\
	&y_{i_0}^{ 1} = y_{i_0}^0 + \frac{\rho}{2}\left(  z^{0} - N\Delta^1  -x_{i_0}^{0}  \right) ;\\
	&\qquad\qquad\qquad\vdots\\
	&x_{i_K}^{K+1} = \frac{1}{2}\left( N\Delta^{K+1} + z^{K } + x_{i_K}^{K}\right) ;\\
	&y_{i_K}^{K+1} = y_{i_K}^{K} + \frac{\rho}{2}\left(   z^{K } - N\Delta^{K+1}  -x_{i_K}^{K}  \right).
	\end{aligned}
	\right.
	\end{equation}
	In (\ref{eq1712}), $\{ z^{k}| k =0,...,K  \}$ are known to the eavesdropper, while $\{x_{i_k}^{k},y_{i_k}^{k},x_{i_k}^{k+1},y_{i_k}^{k+1}|k=0,...,K\}$ are unknown variables to the eavesdropper. Thus, (\ref{eq1712}) consists of $2K+3 $ equations and $2K+2N+2$ unknown variables. Since $N>2$, the eavesdropper cannot solve (\ref{eq1712}) to infer the exact values of $\{x_{i }^{k},y_{i }^{k} |i\in\mathcal{V},k=0, ...,K\}$. Hence with (\ref{eq11}), the gradients $\{\nabla f_i(x_{i }^{k+1})|i\in\mathcal{V},k = 0,...,K\}$ cannot be inferred exactly by the eavesdropper either. 
\end{proof}

The consensus happens when Algorithm 2 converges, where $z$ coincides with local variables $\{x_i\}$. Hence the eavesdropper can use $z$ to infer $\{x_i\}$. However, with finite iterations, the gap between token $z$ and states $\{x_i\}$ can only be guaranteed within a threshold $\epsilon$\footnote[1]{$\epsilon$ can be a predetermined stopping criterion for Algorithm 2.}. Denote $\hat{x}_i$ as the measurement of the eavesdropper regarding to $x_i$. According to (\ref{eq12}), once $x_{i_K}^{K+1}$ is measured, all the states $\{x_{i_K}^{k}|k=0,...,K\}$ of agent $i_K$ can also be estimated sequentially. 
\begin{remark}
	Assume that the eavesdropper uses $z^{K+1}$ to infer $x_{i_K}^{K+1}$ of agent $i_K=K\mod N+1$. If $ \|z^{K+1} -x_{i_K}^{K+1} \|< \epsilon $, there exists an error bound for measurement (\ref{eq1712}) as
	\begin{align}\label{eq191}
	\left\|\hat{x}_{i_K}^{k} - x_{i_K}^{k} \right\| <2^{n}\epsilon ,  ~	\left\|\hat{y}_{i_K}^{k} - y_{i_K}^{k} \right\| < \rho \left(2^{\lfloor \frac{K}{N}\rfloor+1}-2^n \right)\epsilon,
	\end{align}	 
	where $k \in[K-nN+1,K-(n-1)N],1\leq n\leq \lfloor \frac{K}{N}\rfloor$.
\end{remark}
\begin{proof}
	See Appendix C.
\end{proof}

\begin{remark}
	In Algorithm 2, the primal and dual variables, and gradients, i.e., $\{x_i^k,y_i^k, \nabla f_i(x_i^k) | i\in\mathcal{V},k=0,...,K\}$, can be inferred by the eavesdropper asymptoticly, i.e., $K\to \infty$.
\end{remark}
\begin{proof}
	According to \cite{WADMM}, it is guaranteed that $\lim\nolimits_{K\to \infty}z^{K}=x_i^{K},\forall i\in\mathcal{V}$. Furthermore, with equation (\ref{eq12}) and condition $x_i^0 - \frac{y_i^0}{\rho}=\bm{0}$, the values $\{x_i^k,y_i^k|k=0,...,K  \}$ can be derived recursively. Due to (\ref{eq11}), the gradients $\{\nabla f_i(x_i^k) |k=0,...,K  \}$ is revealed.
\end{proof}

As for PI-ADMM1, due to introduced randomization on the update of dual variable, it can guarantee augmented privacy preservation.

\begin{theorem}\label{theorem2}
	In PI-ADMM1, the exact states $\{x_i^{k}| i\in\mathcal{V},k=0,1, ... \}$ cannot be inferred by the eavesdropper unless $z^{k+1}=x_{i_k}^{k+1}$ holds. The multipliers and gradients $\{ y_i^{k}, \nabla f_i(x_i^{k})|i\in\mathcal{N},k=0,1,...  \}$ cannot be exactly inferred either.
\end{theorem}

\begin{proof}
 	 Due to the randomized initialization, the values of $x^0_i$ and $y^0_i$ cannot be revealed by the eavesdropper. Assume that the eavesdropper collects information from $K$ iterations to infer the information of agents. The measurements considering all the agents can be formulated as  
\begin{equation}\label{eq17}
\left\{  
\begin{aligned}
& x_i^0 -\frac{y_i^0}{\rho} = \bm{0},~i\in\mathcal{V};\\
&x_{i_0}^{ 1} = \frac{1}{1 + \gamma_{i_0}^0}\left(   N\Delta^1+ \gamma_{i_0}^0z^{0} + x_{i_0}^{0}\right);\\
&y_{i_0}^{ 1} = y_{i_0}^{0} + \frac{\rho\gamma_{i_0}^0 }{1+ \gamma_{i_0}^0}\left(   z^{0} - N\Delta^1  -x_{i_0}^{0} \right) ;\\
&\qquad\qquad\qquad\vdots\\
&x_{i_K}^{K+1} = \frac{1}{1 + \gamma_{i_K}^K}\left(  N\Delta^{K+1} +\gamma_{i_K}^K z^{K } + x_{i_K}^{K}\right);\\
&y_{i_K}^{K+1} = y_{i_K}^{K} +  \frac{\rho\gamma_{i_K}^K }{1+ \gamma_{i_K}^K}\left(  z^{K } - N\Delta^{K+1}  -x_{i_K}^{K}  \right) .
\end{aligned}
\right.
\end{equation}
Different form (\ref{eq1712}), in the above equations, the tokens $\{ z^{k+1}|  k =0,...,K +1\}$ are known to the eavesdropper, while $\{x_{i_k}^{k},y_{i_k}^{k}|k=0,...,K+1\}$ and $\{ \gamma_{i_k}^k|k=0,...,K \}$ are unknown variables. Thus, (\ref{eq17}) consists of $2K+N+2$ equations and $3K+2N+3$ unknown variables. Thus the eavesdropper cannot solve (\ref{eq17}) to infer the exact values of $\{x_{i }^{k},y_{i }^{k} |i\in\mathcal{V},k=0, ...,K\}$ since $K+N+1>0$.
 When $K\to\infty$, the eavesdropper can have another piece of information according to the KKT conditions (\ref{eq266}). However, due to $K\gg N$, the under-determined system (\ref{eq17}) still cannot be solved asymptotically.
\end{proof}


 \begin{corollary}\label{corollary2}
 	In PI-ADMM1, the local functions $\{f_i|i\in\mathcal{V} \}$ cannot be inferred by the eavesdropper. 
 \end{corollary}
\begin{proof}
	According to Theorem \ref{theorem2}, the states and corresponding gradients are unknown to the eavesdropper. Hence from \cite{zhang}, the local objective functions $\{f_i \}$ of agents cannot be inferred by the eavesdropper.
\end{proof}

From above analysis, we can conclude that the proposed PI-ADMM1 is privacy-preserving for the local functions and datasets of agents against the external eavesdropper. In addition, the privacy for agent $i$ can also be guaranteed against other agents.
\begin{theorem}
	In PI-ADMM1, the states, multipliers and gradients of agent $i(\in\mathcal{V})$, i.e., $\{x_i^k, y_i^k,\nabla f_i(x_i^k)|k=0,1,... \}$, cannot be inferred by other $N-1$ colluding agents.
\end{theorem}
\begin{proof}
 Without loss of generality, we consider the privacy for agent $i_0$ against other agents in set $\mathcal{V}\backslash i_0$, which contains all agents except $i_0$. Recalling the measurements (\ref{eq17}), useful information for deducing the primal and dual variables, and gradients is given by
  \begin{equation}\label{eq18}
 \left\{  
 \begin{aligned}
 & x_{i_0}^0 -\frac{y_{i_0}^0}{\rho} = \bm{0};\\
 &x_{i_0}^{ 1+nN} = \frac{1}{1 + \gamma_{i_0}^{nN}}\left(   N\Delta^{1+nN}+ \gamma_{i_0}^{nN}z^{nN} + x_{i_0}^{nN}\right);\\
 &y_{i_0 }^{ 1+nN} = y_{i_0}^{nN} + \frac{\rho\gamma_{i_0 }^{nN} }{1+ \gamma_{i_0 }^{nN}}\left(   z^{nN} - N\Delta^{1+nN}  -x_{i_0 }^{nN} \right),
 \end{aligned}
 \right.
 \end{equation}
 where $n$ is from $0$ to $ \lfloor \frac{K}{N} \rfloor$. Since (\ref{eq18}) contains $2\lfloor \frac{K}{N} \rfloor + 3$ equations and $3\lfloor \frac{K}{N} \rfloor + 5$ unknown variables, the remaining agents cannot obtain the exact values of $\{x_{i_0}^k,y_{i_0}^k,\nabla f_{i_0}(x_{i_0}^k)|k=0,...,K   \}$ by solving this under-determined system. With $K\to \infty$, the KKT conditions reduce to
 \begin{equation} \label{eq1911}	
\nabla f_{i_0}(x_{i_0}^{K+1}) = y_{i_0}^{K+1};~y_{i_0}^{K+1} = \sum_{j\in\mathcal{V}\backslash i_0 }y_{j}^{K+1};~ z^{K+1} = x_{i_0}^{K+1}. 
\end{equation} 
However, even combining (\ref{eq18}) and (\ref{eq1911}), the exact values of $\{x_{i_0}^k,y_{i_0}^k,\nabla f_{i_0}(x_{i_0}^k)|k=0,...,K   \}$ cannot be obtained asymptotically. 
\end{proof}
  
Since only the token is transmitted among agents in PI-ADMM1, the local variables and gradients $\{x_{i_0}^k,y_{i_0}^k,\nabla f_{i_0}(x_{i_0}^k) \}$ of agent $i_0$ cannot be observed by other agents directly. In fact, except $\{x_j^k,y_j^k,\nabla f_j(x_j^k)|j\in \mathcal{V}\backslash i_0  \}$, the information gathered by the eavesdropper is the same as that of colluding agents in $\mathcal{V}\backslash i_0$. This explains why the system (\ref{eq17}) reduces to (\ref{eq18}) in Theorem 3.

\section{Numerical Experiments}
In this section, we will conduct numerical experiments to evaluate the convergence and privacy properties of proposed PI-ADMM algorithms. For the simulation, we generate the connected network $\mathcal{G}$ with $N $ agents and $|\mathcal{E}|=\frac{N(N-1)}{2}\eta$ links. This ensures a Hamiltonian cycle in $\mathcal{G}$. We consider unicast among agents, and the resultant communication cost for each transmission of a $p$-dimensional vector is 1 unit. 

We evaluate the convergence of proposed approaches with state-of-the-art approaches regarding the accuracy defined by  
\begin{equation}\label{acc}
\text{accuracy}=\frac{1}{N}\sum_{i=1}^{N}\frac{\left\|x_i^k-x^*\right\|}{\left\|x_i^0-x^*\right\|},
\end{equation} 
where $x^*\in\mathbb{R}^p$ is the optimal solution of (\ref{eq2}). The dimension of $x_i$, $y_i$ and $z$ is set to be $p=2$.

   \begin{figure} [t] 
  	\vskip 0.2in
  	\begin{center}
  		\centerline{\includegraphics[width=86mm]{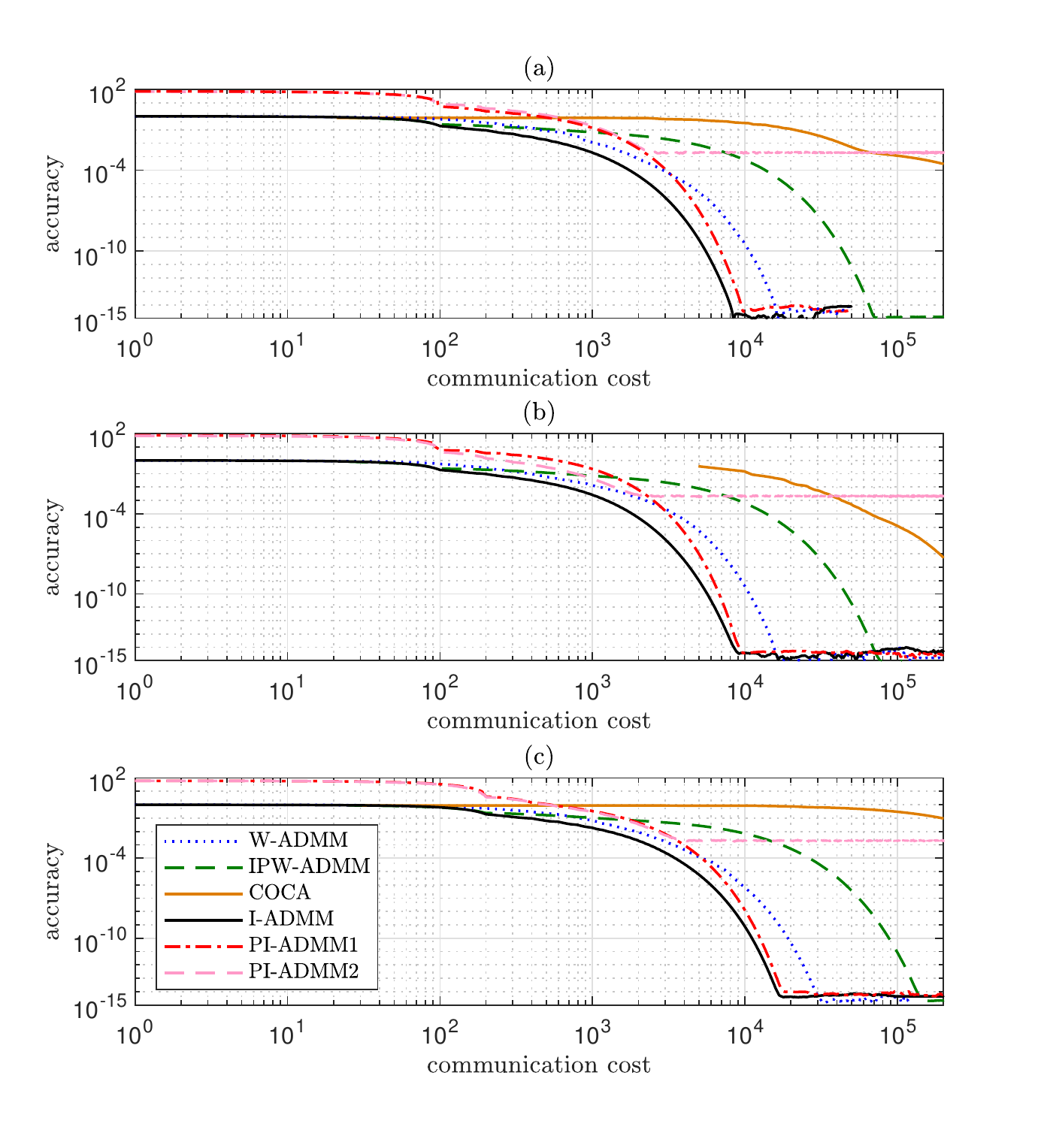}}
  		\caption{The accuracy of \textit{decentralized ridge regression} for W-ADMM ($\beta=10$), IPW-ADMM ($\rho=1,\tau=0,M=25$), COCA ($c=1,\alpha=1,\rho=0.85$), I-ADMM ($\rho=10$), PI-ADMM1 ($\rho=10$) and PI-ADMM2 ($\rho=10,\sigma=10^{-3}$) with different network settings: (a) $N=100,\eta=0.3$; (b) $N=100,\eta=0.5$; (c) $N=200,\eta=0.3$.}
  		\label{2333}
  	\end{center}
  	\vskip -0.2in
  \end{figure} 
\subsection{Decentralized Ridge Regression}   
  
We first consider the decentralized least square problem as \cite{8523816}, which aims at solving (\ref{eq1}) with a local function
\begin{equation}
	f_i(x_i;\mathcal{D}_i) = \frac{1}{b_i}\sum_{j=1}^{b_i} \left\| x_i^To_{i,j}-t_{i,j} \right\|^2,
\end{equation}
where $\mathcal{D}_i=\{o_{i,j},t_{i,j}|j=1,...,b_i \}$ is the dataset of agent $i$ locally. The entries of input $o_{i,j}\in\mathbbm{R}^{2}$ and target $t_{i,j}\in\mathbbm{R}$ follow i.i.d. distribution $\mathcal{U}(0,1)$. The number of data samples is kept unique across agents with $b_i=30$. 
 
   \begin{figure} [t] 
	\vskip 0.2in
	\begin{center}
		\centerline{\includegraphics[width=86mm]{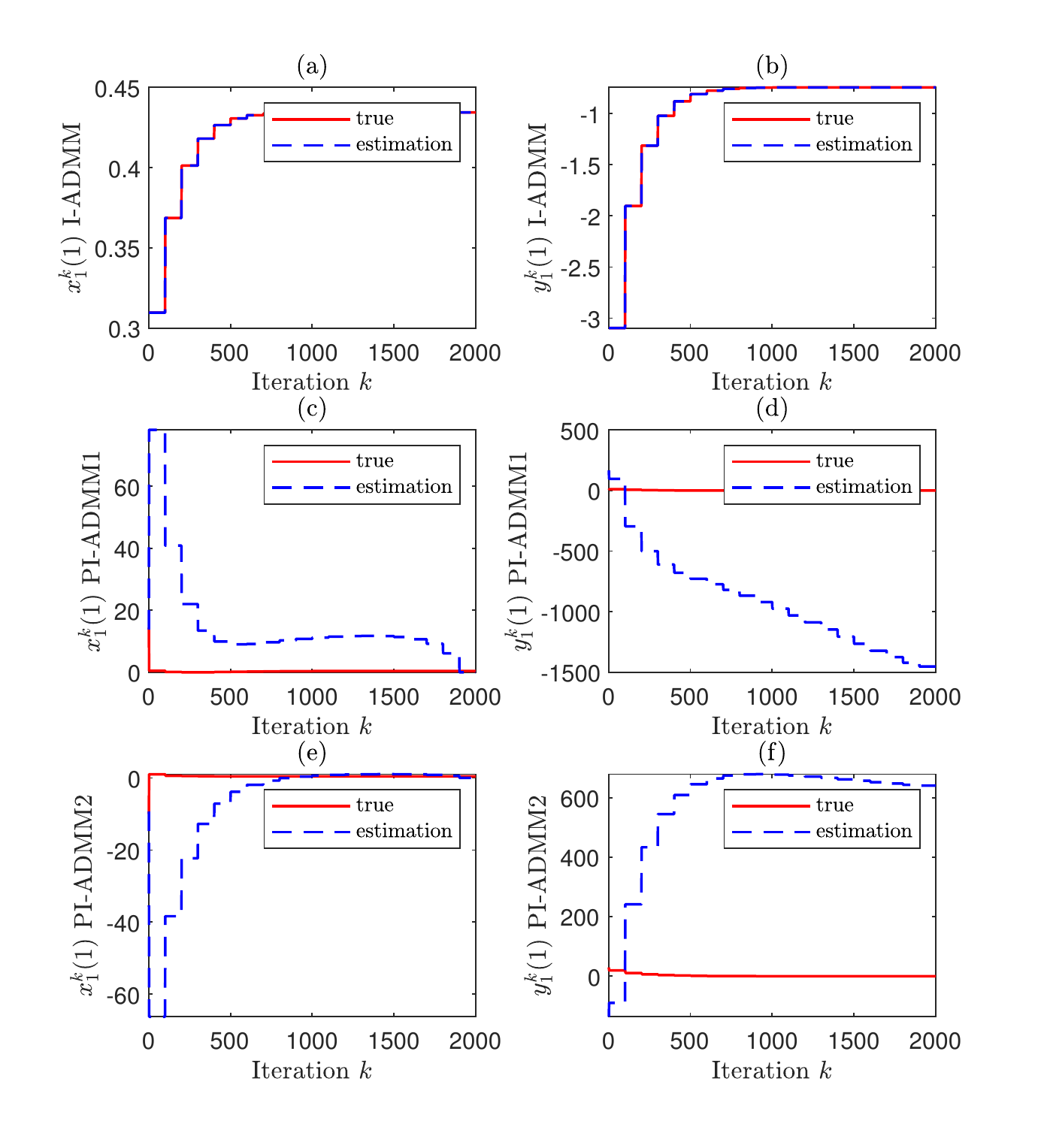}}
		\caption{The estimation and true value in \textit{decentralized ridge regression} with respect to $\{ x_{1}^k(1),y_{1}^k(1)|k=0,1,...,2000 \}$ for I-ADMM, PI-ADMM1 and PI-ADMM2. }
		\label{23}
	\end{center}
	\vskip -0.2in
\end{figure} 

The accuracy over communication cost is shown in Fig. 2. It is clear that I-ADMM is the most communication efficient compared with W-ADMM \cite{WADMM}, IPW-ADMM \cite{PWADMM} and COCA \cite{COCA} for different network settings. Equivalently, it demonstrates that I-ADMM has a faster convergence speed than W-ADMM. This is due to that both I-ADMM and W-ADMM consume the same communication load with same amount of iterations. The update order of I-ADMM follows Hamiltonian cycle, while W-ADMM is updated by a random walk. In W-ADMM, each agent is randomly visited, and thus the unbalanced updating frequency of agents degrades the convergence speed. It also can be demonstrated that I-ADMM is more communication efficient than W-ADMM. Besides, PI-ADMM1 presents different convergence behaviors compared with I-ADMM and W-ADMM. To guarantee the privacy in PI-ADMM1, we set $\{\mathbbm{P}_i=\mathcal{U}(1-\frac{1}{\rho},1+\frac{1}{\rho})|i\in\mathcal{V} \}$, and generate $\{x_{i }^0\sim \mathcal{U}(0,100)|i \in\mathcal{V} \}$, which hence makes the initialization $\{x_i^0\}$ far away from the optimal solution $x^*$. This explains the phenomenon that PI-ADMM1 converges slower than I-ADMM and W-ADMM with few iterations. However, along with the updating process, the convergence speed of PI-ADMM1 surpasses that of W-ADMM, which coincides with that of I-ADMM. This is because the artificial noise added on the step size is scaled by the primal residue $z^{k+1}-x_i^{k+1}$, which converges to $\bm{0}$ gradually. Since the additive noise in PI-ADMM2 does not scale with iterations. As shown in Fig. 4, this introduces error bounds for the accuracy, which is determined by $\sigma$. Comparing sub-figures (a) and (b), the convergence behaviors of I-ADMM, PI-ADMM1 and PI-ADMM2 do not depend on the connectivity $\eta$. This is because the convergence property of these proposed algorithms is only determined by the size of Hamiltonian cycle. This also explains that the convergence speed degrades with the expanding networks.

\begin{figure} [t] 
	\vskip 0.2in
	\begin{center}
		\centerline{\includegraphics[width=86mm]{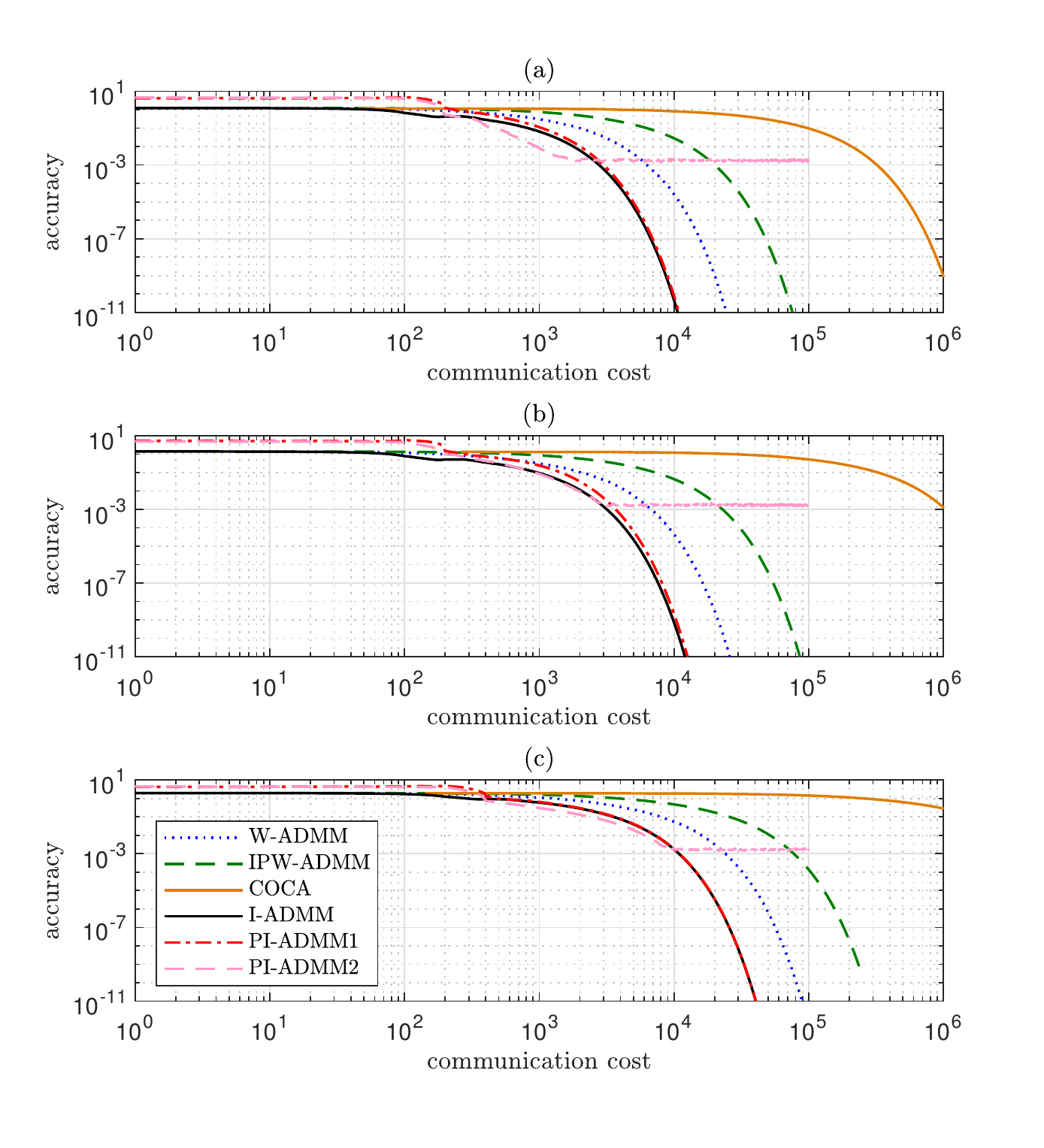}}
		\caption{The accuracy of \textit{decentralized logistic regression} for W-ADMM ($\beta=1$), IPW-ADMM ($\rho=1,\tau=3,M=25$), COCA ($c=1,\alpha=1,\rho=0.85$), I-ADMM ($\rho=1$), PI-ADMM1 ($\rho=1$) and PI-ADMM2 ($\rho=1,\sigma=10^{-3}$) with different network settings: (a) $N=100,\eta=0.3$; (b) $N=100,\eta=0.5$; (c) $N=200,\eta=0.3$. }
		\label{233}
	\end{center}
	\vskip -0.2in
\end{figure} 
 
Then we evaluate the capability of privacy preservation for proposed PI-ADMM algorithms with $N=100$ and $\eta=0.3$. Without loss of generality, we only focus on the estimation over $\{x_{1}^k(1),y_{1}^k(1)|k=0,1,...,2000\}$ with observations $\{z^k|k=0,1,...,2001\}$ for proposed methods. Besides, as for I-ADMM, we assume that the eavesdropper uses the recursive equation (\ref{eq12}) to deduce the primal and dual variable of agent $1$ with initialization $\{x_i=\bm{0},y_i=\bm{0}|\in\mathcal{V}\}$ and inferring $x_{i_{2000}}^{2001}$ with $z ^{2001}$. The results are shown in Fig. 6 (a) and (b), where the estimation coincides with the true value. The figures clearly demonstrates that I-ADMM is not privacy preserving. While for PI-ADMM1, we reformulate the under-determined equations (\ref{eq17}) into the form $A\upsilon=b$, where $\upsilon=[x_1^0(1),...,x_1^{2000}(1),y_1^0(1),...,y_1^{2000}(1)]$. $A\in\mathbbm{R}^{l\times m}$ is the coefficient matrix that also considers the KKT condition (\ref{eq266}) as $\sum_{i=1}^{N}y_i^{2000}(1)= 0$. Since $\{\gamma_{i_k}^k \}$ is unknown to the eavesdropper, we set $\gamma_{i_k}^k=1$ in the estimation without loss of generality. Moreover, we substitute $x_{i_{2000-i}}^{2001-i}$ with $z^{2001-i}$ for $i=0,1,...,N-1$ both in systems (\ref{eq1712}) and (\ref{eq17}). Since $A$ is sparse and $l>m$, we use core function \textit{lsqr}($A,b$) in matlab to obtain the solution $\upsilon^*$, which solves the least square problem
\begin{equation}
	\upsilon^* = \arg\min_{\upsilon}~\|A\upsilon-b\|^2.
\end{equation}
A similar procedure for estimation can also be established for PI-ADMM2. The estimation results for PI-ADMM1 and PI-ADMM2 are presented in Fig. 6 (c)-(f). From sub-figures (c) and (e), for both PI-ADMM1 and PI-ADMM2, the gaps between the true value and estimation for $x_{1}^k(1)$ shrinkages to $0$ as convergence happens. This is inevitable and also consistent with analysis in Remark 4. However, the estimation error for $y_{1}^k(1)$ diverges with iterations. We can see that PI-ADMM1 and PI-ADMM2 can prevent information leakage against the external eavesdropper regarding to the multipliers. 
 
\begin{figure} [t] 
	\vskip 0.2in
	\begin{center}
		\centerline{\includegraphics[width=86mm]{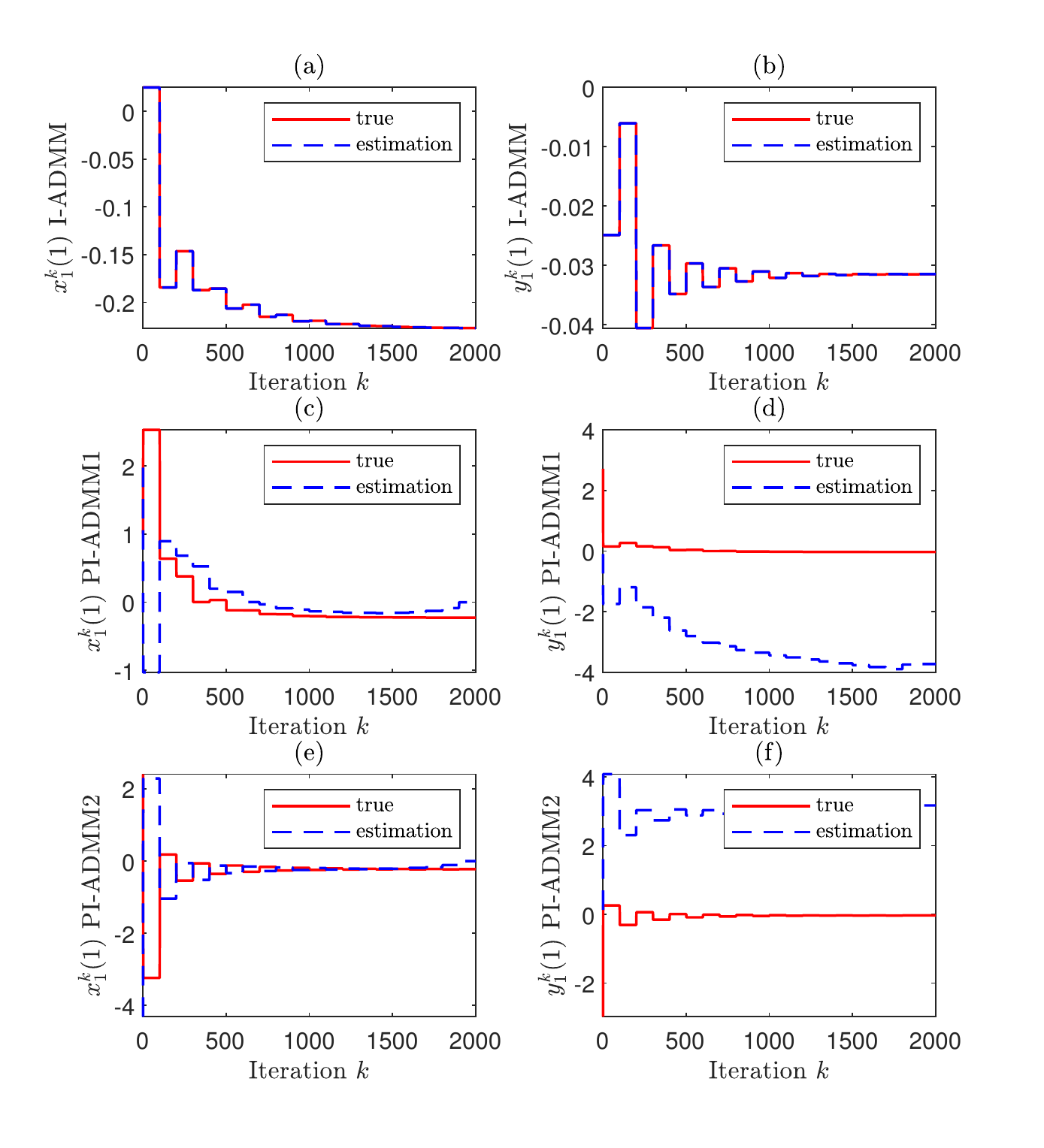}}
		\caption{The estimation and true value in \textit{decentralized logistic regression} with respect to $\{ x_{1}^k(1),y_{1}^k(1)|k=0,1,... 2000\}$ for I-ADMM, PI-ADMM1 and PI-ADMM2.}
		\label{2331}
	\end{center}
	\vskip -0.2in
\end{figure}

\subsection{Decentralized Logistic Regression}
In the decentralized logistic regression, the local loss function of agent $i$ is
\begin{equation}\label{eq22}
f_i(x_i) = \frac{1}{b_i}\sum_{j=1}^{b_i}  \log\left( 1+\exp\left(-t_{i,j}x_i^To_{i,j}\right)\right) ,
\end{equation}
where $t_{i,j}\in\{-1,1\}$ and $b_i=30$. Each sample feature $o_{i,j}$ follows $\mathcal{N}(0,I)$. To generate $t_{i,j}$, we first generate a random vector $x\in\mathbbm{R}^2\sim\mathcal{N}(0,I)$. Then for each sample, we generate $v_{i,j}$ according to $\mathcal{U}(0,1)$, and if $v_{i,j}\leq (1+\exp(-x^To_{i,j}))^{-1}$, we set $t_{i,j}$ as $1$, otherwise $-1$. Since it is difficult to solve the optimization problem (\ref{eq2}) with (\ref{eq22}) in I-ADMM based methods, we alternatively use the updating process (\ref{eq4b11}).
Fairly, we also adopt the first-order approximation for algorithms IPW-ADMM, W-ADMM and COCA.

In Fig. 5, we present the accuracy over communication costs for different network settings. Apparently, compared to W-ADMM, IPW-ADMM and COCA, the proposed I-ADMM based algorithms are the most communication-efficient. The curves with different network setups present the similar trends as those of Fig. 3. 
 
The results regarding to privacy preservation are shown in Fig. 6. For PI-ADMM1 and PI-ADMM2, we set $\{\mathbbm{P}_i=\mathcal{U}(1-\frac{0.1}{\rho},1+\frac{0.1}{\rho})|i\in\mathcal{V} \} $ and initialize $\{x_{i }^0\sim \mathcal{U}(0,10)|i \in\mathcal{V} \}$. The estimation approaches for I-ADMM, PI-ADMM1 and PI-ADMM2 follow those of Fig. 4. In Fig. 6, the sub-figures (a) and (b) demonstrate that I-ADMM cannot guarantee privacy for agents. As for PI-ADMM1 and PI-ADMM2, the states $ x_1^{k}(1) $ can be revealed with iteration $k$ increasing. However, the estimation for $y_1^{k}(1)$ diverges from the ground truth. Thus, proposed PI-ADMM algorithms are privacy preserving.

\section{Conclusions}

We investigate the I-ADMM based privacy preserving methods for solving consensus problem in decentralized networks. Different from W-ADMM, the updating order of I-ADMM follows an Hamiltonian cycle, that guarantees a faster convergence speed than W-ADMM. Since fewer iteration rounds are needed, I-ADMM is more communication efficient than W-ADMM and other state-of-the-art methods. Moreover, to prevent information leakage over the local function and private dataset against the external eavesdropper, we propose two PI-ADMM algorithms, i.e., PI-ADMM1 and PI-ADMM2, which conduct step size perturbation and primal perturbation, respectively. We prove the convergence and privacy preservation for PI-ADMM1. Numerical experiments also show that PI-ADMM2 is privacy preserving. Besides, numerical results show that PI-ADMM1 can achieve almost the same communication efficiency as I-ADMM.



 \begin{appendices}

 	\section{Proof of Lemma \ref{lemma2}}
 	From steps 2-11 of PI-ADMM1, it is easy to verify that $ \frac{1}{N}\sum_{i=1}^{N} (x_i^0 -\frac{y_i^0}{\rho} ) =\bm{0}$. Then we prove that $\mathcal{L}_{\rho}(\bm{x}^k,\bm{y}^k,z^{k })$ is non-increasing with iteration $k$. From the optimality condition of (\ref{eq4b}) and update (\ref{eq4c1}), we can derive
 	\begin{equation} 
 	\begin{aligned}
 	\nabla f_{i_k}\left(x_{i_k}^{k+1}\right)=&y_{i_k}^k + \rho\gamma_{i_k}^k \left(z^{k }-x_{i_k}^{k+1}  \right)=y_{i_k}^{k+1} . \label{eq131}
 	\end{aligned}	 	
 	\end{equation}
 	After updating to $\bm{x}^{k+1}$ and $\bm{y}^{k+1}$ by (\ref{eq4a}), (\ref{eq4c1}), we have
 	\begin{equation}
 	\begin{aligned}
 	&\mathcal{L}_{\rho}\left(\bm{x}^k,\bm{y}^k,z^{k } \right) - \mathcal{L}_{\rho}\left( \bm{x}^{k+1},\bm{y}^{k+1},z^{k } \right) \\
 	\overset{(a)}{=}&\underline{f_{i_k}\left(x_{i_k}^k\right)-f_{i_k}\left(x_{i_k}^{k+1}\right)}_{A} + \underline{\left\langle y_{i_k}^k,z^{k }-x_{i_k}^k\right \rangle }   \\
 	&\underline{-\left \langle y_{i_k}^{k+1},z^{k }-x_{i_k}^{k+1}  \right\rangle +\left\langle \rho\left(z^{k }-x_{i_k}^{k+1}\right),x_{i_k}^{k+1} -x_{i_k}^k \right\rangle}_{B}\\&+ \frac{\rho}{2} \left\| x_{i_k}^k-x_{i_k}^{k+1}  \right\|^2,\label{eq14} 
 	\end{aligned}
 	\end{equation}
 	where $(a)$ holds since the cosine identity $\|b+c\|^2 - \|a+c\|^2=\|b-a \|^2 + 2\langle a+c,b-a \rangle $. According to Lemma 1, term $A$ in (\ref{eq14}) satisfies
 	\begin{equation}
 	\begin{aligned}
 	A &\geq \left\langle \nabla f_{i_k}\left(x_{i_k}^{k }\right), x_{i_k}^k-x_{i_k}^{k+1}\right \rangle -\frac{L}{2}\left\|x_{i_k}^{k+1}-x_{i_k}^k  \right\|^2 \\
 	&=\left\langle y_{i_k}^{k}, x_{i_k}^k-x_{i_k}^{k+1} \right\rangle   -\frac{L}{2}\left\|x_{i_k}^{k+1}-x_{i_k}^k  \right\|^2 ,
 	\end{aligned}
 	\end{equation}
Then $A + B$ can be rewritten as
 	\begin{equation} \label{eq21}
 	\begin{aligned}
  &A+B \\ \overset{(a)}{=}&\left\langle y_{i_k}^k - y_{i_k}^{k+1},z^{k }-x_{i_k}^{k+1} \right\rangle  +\frac{1}{\gamma_{i_k}^k}\left\langle  y_{i_k}^{k+1}- y_{i_k}^k  ,x_{i_k}^{k+1} -x_{i_k}^k \right\rangle \\
 	\overset{(b)}{\geq}&-\frac{1}{\gamma_{i_k}^k}\left(\frac{1}{\rho } + \frac{1}{2}\right)\left\| y_{i_k}^{k+1}- y_{i_k}^{k} \right\|^2 -\left(\frac{1}{2\gamma_{i_k}^k}+\frac{L}{2}\right)\left\|x_{i_k}^{k+1}- x_{i_k}^{k} \right\|^2 ,
 	\end{aligned}
 	\end{equation}
 	where (a) is because (\ref{eq20}) and (b) are from Young's inequality.
 	 	\begin{equation}\label{eq20}
 	\rho\left(z^{k }-x_{i_k}^{k+1}\right)=\frac{1}{\gamma_{i_k}^k}\left( y_{i_k}^{k+1}-y_{i_k}^{k} \right).
 	\end{equation}
 Through updating token $z$ at iteration $k$, the change of augmented Lagrangian is measured by
 	\begin{equation}
 	\begin{aligned}
 	&\mathcal{L}_{\rho}\left(\bm{x}^{k+1},\bm{y}^{k+1},z^{k } \right) - \mathcal{L}_{\rho}\left( \bm{x}^{k+1},\bm{y}^{k+1},z^{k+1} \right) \\
 	=&\sum_{i=1}^{N}\left[ \vphantom{\frac{\rho}{2}}  \left\langle y_i^{k+1}, z^{k }-z^{k+1}\right\rangle  \right.\\& ~~~~~  +  \left.\frac{\rho}{2}\left(\left\|z^{k }-x_i^{k+1}\right\|^2-  \left\|z^{k+1}-x_i^{k+1}\right\|^2   \right)\right] \\
 	=&\frac{\rho N}{2} \left\| z^{k+1}-z^{k }\right\|^2 + \sum_{i=1}^{N}\rho\left\langle z^{k+1}-x_i^{k+1}+\frac{y_i^{k+1}}{\rho},z^{k }-z^{k+1} \right \rangle \\
 	\overset{(a)}{=}&\frac{\rho N}{2} \left\| z^{k+1}-z^{k }\right\|^2,\label{eq172}
 	\end{aligned}
 	\end{equation}
 	where $(a)$ holds because the truth $z^{k+1}=\frac{1}{N}\sum_{i=1}^{N}(x_i^{k+1}-\frac{y_i^{k+1}}{\rho})$. By combining (\ref{eq14}) and (\ref{eq172}), we obtain
 	\begin{equation}\label{eq23}
 	\begin{aligned}
 	&\mathcal{L}_{\rho}\left(\bm{x}^k,\bm{y}^k,z^{k } \right) - \mathcal{L}_{\rho}\left( \bm{x}^{k+1},\bm{y}^{k+1},z^{k+1} \right)\\
 	=&\left(\frac{\rho-L}{2}-\frac{1}{2 \gamma_{i_k}^k} \right) \left\| x_{i_k}^k-x_{i_k}^{k+1}  \right\|^2+ \frac{\rho N}{2} \left\| z^{k+1}-z^{k }\right\|^2 \\&-\frac{1}{\gamma_{i_k}^k}\left(\frac{1}{\rho } + \frac{1}{2}\right)\left\| y_{i_k}^{k+1}- y_{i_k}^{k} \right\|^2  \\
 	\overset{(a)}{\geq} &\rho N\left( \frac{1}{2}-\frac{\rho N+2N}{\gamma_{i_k}^k} \right)\left\|z^{k+1}- z^{k} \right\|^2 \\& +\left(\frac{\rho-L}{2}-\frac{1}{2 \gamma_{i_k}^k} -\frac{\rho^2 + 2\rho }{\gamma_{i_k}^k} \right)\left\| x_{i_k}^{k+1}-x_{i_k}^k \right\|^2  ,
 	\end{aligned}
 	\end{equation}
 	where $(a)$ is from $y_{i_k}^{k+1}-y_{i_k}^k=\rho(x_{i_k}^{k+1}-x_{i_k}^k)-\rho N (z^{k+1}-z^k)$ and triangle inequality. Hence to ensure that $\mathcal{L}_{\rho}(\bm{x}^k,\bm{y}^k,z^{k} )$ is non-increasing with iterations, we must have $\rho>L$ and $\gamma_{i_k}^k> \max\{\frac{2\rho^2 + 4\rho +1}{\rho-L} ,2(\rho+2)N  \}$. This proves property 1). To prove statement 2), we verify that $\mathcal{L}_{\rho}(\bm{x}^{k+1},\bm{y}^{k+1},z^{k+1})$ is lower bounded.
 \begin{align}\label{eq19}
&\mathcal{L}_{\rho}\left(\bm{x}^{k+1},\bm{y}^{k+1},z^{k+1}\right) \notag\\
=& \sum_{j=0}^{N-1} \left[\vphantom{\left\|x_{i_{k-j}}^{k-j+1}  \right\|^2}  f_{i_{k-j}}\left(x_{i_{k-j}}^{k-j+1}\right)+\left\langle y_{i_{k-j}}^{k-j+1 }, z^{k+1}-x_{i_{k-j}}^{k-j+1} \right\rangle\right.    \notag\\&\left.~~~~~ +\frac{\rho}{2}\left\|z^{k+1}-x_{i_{k-j}}^{k-j+1}  \right\|^2 \right]   \notag\\ 
\overset{(a)}{\geq}&\sum_{j=0}^{N-1}\left[f_{i_{k-j}}\left(z^{k+1}\right)+\frac{\rho-L}{2}\left\|z^{k+1}-x_{i_{k-j}}^{k-j+1} \right\|^2  \right. \notag\\&~~~~~~ \left. -\left\langle \nabla f_{i_{k-j}}\left(x_{i_{k-j}}^{k-j+1}\right)-y_{i_{k-j}}^{k-j+1 },   z^{k+1}-x_{i_{k-j}}^{k-j+1} \right\rangle \right]   \notag\\
\geq &  \min_{x}\left\{\sum_{i=1}^{N} f_{i}(x)\right\} +  \frac{\rho-L}{2}\sum_{j=0}^{N-1}  \vphantom{\frac{1}{\gamma_{i_{k-j}}^{k-j}}}\left  \|z^{k+1 }-x_{i_{k-j}}^{k-j+1} \right\|^2   \notag\\\overset{(b)}{\geq}& F(x^*)>-\infty,   
\end{align}
 	where $(a)$ is from L-Lipschitz differentiable $f_i$ and $(b)$ is from the Assumption 2 in which the optimal value $F(x^*)$ is assumed to be lower bounded. Guaranteeing $\rho>L$ can make (\ref{eq19}) satisfied, which completes the proof for property 2).
Then with the monotonicity statement 1), we conclude that $\mathcal{L}_{\rho}$ is convergent.

 \section{Proof of Theorem \ref{theorem111}}
 Note that the KKT point $(\bm{x}^*,\bm{y}^*,z^*)$ of problem (\ref{eq2}) satisfies the following conditions
 \begin{subequations}\label{eq266}
 	\begin{align}		
 	&\nabla f_i\left(x_i^*\right) - y_i^*= \bm{0},~ \forall i\in\mathcal{V};\\
 	&\sum_{i=1}^{N}y_i^* = \bm{0}; \\
 	&z^* = x_i^*,~\forall i\in\mathcal{V}. 
 	\end{align}	
 \end{subequations}
 Since (\ref{eq266}) also implies that $\sum_{i=1}^{N}\nabla f_i(x_i^*)=\bm{0}$, $\bm{x}^*$ is also a stationary point of problem (\ref{eq2}). By summing (\ref{eq23}) from all iterations $0$ to $k$, we obtain
\begin{equation}
 \begin{aligned}\label{eq25}
 &\sum_{l=0}^{k}\left[\rho N\left( \frac{1}{2}-\frac{\rho N +2 N }{\gamma_{i_l}^l} \right)\left\|z^{l+1}- z^{l} \right\|^2 \right.  \\& ~~~~~ \left.+\left(\frac{\rho-L}{2}-\frac{1}{2 \gamma_{i_l}^l} -\frac{\rho^2 + 2\rho }{\gamma_{i_l}^l} \right)\left\| x_{i_l}^{l+1}-x_{i_l}^l \right\|^2 \right]  \\
 \leq &\mathcal{L}_{\rho}\left(\bm{x}^0,\bm{y}^0,z^0\right)-\mathcal{L}_{\rho}\left(\bm{x}^{k+1},\bm{y}^{k+1},z^{k+1}\right) \\
 =&\left[ \mathcal{L}_{\rho}\left(\bm{x}^0,\bm{y}^0,z^0\right)-F(x^*) \right]-\left[\mathcal{L}_{\rho}\left(\bm{x}^{k+1},\bm{y}^{k+1},z^{k+1}\right)-F(x^*) \right] \\
 \overset{(a)}{\leq}& \mathcal{L}_{\rho}\left(\bm{x}^0,\bm{y}^0,z^0\right)-F(x^*) <+\infty,
 \end{aligned}
 \end{equation}
 where (a) is from (\ref{eq19}). Then with the conditions $\rho>L$ and $\gamma_{i_k}^k> \max\{\frac{2\rho^2 + 4\rho +1}{\rho-L} ,2(\rho+2)N \}$, the left hand side (LHS) of inequality (\ref{eq25}) is positive and increasing with $k$. Since the RHS of (\ref{eq25}) is finite, the iterates $(\bm{x}^k,\bm{y}^k,z^k)$ must satisfy,  
 
 \begin{equation}\label{eq26}
 \lim\limits_{k\to \infty}	\left\| z^{k+1}-z^{k}\right\|=0;~\lim\limits_{k\to \infty}	\left\| x_i^{k+1}-x_i^{k}\right\|=0,~\forall i\in\mathcal{V}.
 \end{equation}
 In addition, from the inequality of (\ref{eq21}), the dual residue is given by $\lim\nolimits_{k\to\infty}  \|y_{i_k}^{k+1}-y_{i_k}^{k }  \| \leq \lim\limits_{k\to\infty} \rho  \|x_{i_k}^{k+1}-x_{i_k}^k  \| + \rho N \| z^{k+1} -z^k \|=0 $. Since $ \|y_{i_k}^{k+1}-y_{i_k}^{k }   \|\geq 0$, we can conclude that  
 \begin{equation}\label{eq27}
 \lim\limits_{k\to\infty} \left\|y_{i }^{k+1}-y_{i}^{k }  \right\|=0  ,~\forall i\in\mathcal{V}.
 \end{equation}
 Then we utilize (\ref{eq26}) and (\ref{eq27}) to show that the limit point of $(\bm{x}^k,\bm{y}^k,z^k)$ is a KKT point of problem (\ref{eq2}). For any $i\in\mathcal{V}$, there always exists a $j\in[0,N-1]$ satisfying 
 \begin{equation}
 \begin{aligned}
 \left\|z^{k+1}-x_i^{k+1}   \right\|=& \left\|z^{k}-x_{i_{k-j}}^{k-j+1}   \right\|\\
 \leq & \left\|z^{k+1}-z^{k-j}   \right\|  + \left\| z^{k-j}-x_{i_{k-j}}^{k-j+1}   \right\|\\
 \overset{(a)}{=}&\left\|z^{k+1}-z^{k-j}   \right\| + \frac{1}{\rho \gamma_{i_{k-j}}^{k-j}}\left\|  y_{i_{k-j}}^{k-j+1 } -y_{i_{k-j}}^{k-j }\right\|,
 \end{aligned}
 \end{equation} 
 where (a) is from (\ref{eq20}). Then we can conclude 
 \begin{equation}\label{eq31}
 \lim\limits_{k\to\infty}\left\|z^{k+1}-x_i^{k+1}  \right\|=0,\forall i\in\mathcal{V}.
 \end{equation}  
 By applying (\ref{eq4a}) and (\ref{eq131}), we obtain that
 \begin{equation}\label{eq32}
 \lim\limits_{k\to\infty}\sum_{i=1}^{N}y_i^k=\bm{0};~   \lim\limits_{k\to\infty}\nabla f_i(x_i^k) - y_i^k = \bm{0},~\forall i\in\mathcal{V}.
 \end{equation} 
 Equations (\ref{eq31}) and (\ref{eq32}) imply that $(\bm{x}^k,\bm{y}^k,z^k)$ asymptotically satisfy the KKT conditions in (\ref{eq266}). Then we conclude that the iterates $(\bm{x}^k,\bm{y}^k,z^k)$ in PI-ADMM1 have limit points, which satisfy KKT conditions of original problem (\ref{eq2}).
 
  	\section{Proof of Remark 3}
  According to (\ref{eq4b}), the states of agent $i_K$ satisfy
  \begin{align}
  x_{i_K}^{K-nN+1}=x_{i_K}^{K-nN+2}=...=x_{i_K}^{K-(n-1)N}, 1\leq n\leq  \left\lfloor \frac{K}{N}\right\rfloor.
  \end{align}	 
  Then from recursive equation (\ref{eq12}), for $k \in[K-nN+1,K-(n-1)N] $, we have
  \begin{align} 
  x_{i_K}^{k} =&x_{i_K}^{K-(n-1)N} \notag\\ =&2 x_{i_K}^{K-(n-1)N+1} - N\Delta^{K-(n-1)N+1}  - z^{K-(n-1)N+1}\notag \\
  =&4 x_{i_K}^{K-(n-2)N+1} - 2N\Delta^{K-(n-1)N+1}  - 2z^{K-(n-1)N+1}\notag\\&
  - N\Delta^{K-(n-2)N+1}  - z^{K-(n-2)N+1} \notag\\
  &\qquad\qquad\qquad\vdots\notag \\
  =&2^{n}x_{i_K}^{K +1} - \sum_{j=1}^{n}\left( 2^{n-j } N \Delta^{K-(n-j)N+1}  -  2^{ n-j} z^{K-(n-j)N+1} \right).\label{eq171} 
  \end{align}		
  By substituting $x_{i_K}^{K +1}$ with $z^{K+1}$ in (\ref{eq171}), the measurement $\hat{x}_{i_K}^{k}$ is obtained. With the condition $ \|z^{K+1} -x_{i_K}^{K+1} \|< \epsilon $, the LHS of (\ref{eq191}) is obtained. Since the eavesdropper cannot deduce $y_{i_K}^K$ directly, instead $y_{i_K}^k$ can be derived recursively from $y_{i_K}^0$, where 
  \begin{equation}
  \begin{aligned}
  y_{i_K}^{k} =&y_{i_K}^{K-nN+1}  \\
  =&y_{i_K}^{K-nN}+\frac{\rho}{2} \left(  z^{K-nN } - N \Delta^{K-nN+1} -x_{i_K}^{K-nN }\right)\\
  &\qquad\qquad\qquad\vdots \\
  =&y_{i_K}^0 + \sum_{j=n}^{\lfloor  \frac{K}{N} \rfloor}\frac{\rho}{2} \left(  z^{K-jN } - N \Delta^{K-jN+1} -x_{i_K}^{K-jN }\right).
  \end{aligned}			
  \end{equation}
  Due to $y_{i_K}^0 = \rho x_{i_K}^0$ and the bounds for measurement $\hat{x}_{i_K}^k$, the final result is concluded.
 
 \end{appendices}

\bibliography{ref}
\bibliographystyle{IEEEtran}

\end{document}